\documentclass{smfart}
\usepackage{amsmath, amssymb, amsthm, amscd,smfenum,xspace,mathrsfs,url,upref,verbatim}
\usepackage[utf8]{inputenc} 
\usepackage[T1]{fontenc}
\usepackage{dsfont}

\usepackage{stmaryrd}

\usepackage[all,cmtip]{xy}

\usepackage{ textcomp }

\let\psi=\Psi
\let\hat=\widehat
\let\tilde=\widetilde
\numberwithin{equation}{subsection}

\newtheorem{theorem}[equation]{Théorème} 

\newtheorem{lemme}[equation]{Lemme}
\newtheorem{corollaire}[equation]{Corollaire}

\theoremstyle{remark}

\newtheorem{definition}[equation]{Définition}

\DeclareMathOperator{\Char}{Char}

\DeclareMathOperator{\DR}{DR}

\DeclareMathOperator{\ord}{ord}

\DeclareMathOperator{\Spec}{Spec}

\def\cartesien{\ar@{}[rd]|{\square}}

\DeclareMathOperator{\Ext}{Ext}

\DeclareMathOperator{\Gr}{Gr}

\DeclareMathOperator{\an}{an}
\DeclareMathOperator{\End}{End}
\DeclareMathOperator{\rg}{rg}

\usepackage[all,cmtip]{xy}

\title[Un analogue pour les connexions d'une construction d'Abbes et Saito ]{Un théorème de linéarité de la construction d'Abbes et Saito pour les connexions méromorphes. 
}

\author[J.-B.~Teyssier]{Jean-Baptiste Teyssier}
\date{}
\curraddr{Freie Universität Berlin, Mathematisches Institut, Arnimallee 3, 14195 Berlin, Germany}
\email{teyssier@zedat.fu-berlin.de}
\usepackage[french,english]{babel}
\begin{document}
\maketitle
\section*{Introduction}
Ce travail s'inscrit dans le contexte des analogies entre la ramification sauvage des faisceaux $\ell$-adiques et l'irrégularité des connexions méromorphes en caractéristique $0$. \\ \indent 
Si $X$ est une variété sur un corps parfait de caractéristique $p>0$, $D$ un diviseur à croisements normaux de $X$ et $U= X\setminus D$, Abbes et Saito ont dégagé dans \cite{Saito} et \cite{clean} une procédure de spécialisation le long de $D$ pour les faisceaux $\ell$-adiques constructibles lisses sur $U$ produisant une mesure géométrique de la ramification sauvage le long de $D$. Dans le cas d'un trait \cite{AS}, le receptacle de leur invariant est une droite vectorielle admettant une interprétation différentielle, et l'invariant en question un ensemble fini de points fermés de cette droite. En dimension supérieure, ce receptacle devient un fibré $T$ au-dessus de $D$ et sous une hypothèse de contrôle de la ramification en tout point de $D$, le spécialisé d'Abbes et Saito est un système local sur $T$ de transformé de Fourier à support fini au-dessus de $D$. \\ \indent
Cette propriété de finitude tient à ce qu'après restriction à une fibre de $T$, le spécialisé d'Abbes et Saito s'identifie à une somme directe finie d'images inverses par des formes linéaires du faisceau d'Artin-Schreier sur la droite affine. C'est la propriété \textit{d'additivité} \cite[3.1]{clean}. 
\\ \indent
Pour les connexions méromorphes, la spécialisation d'Abbes et Saito fait sens et on démontre dans ce texte\footnote{Voir \ref{conj}.} que si $D$ est un hyperplan de l'espace affine $\mathds{A}^{n}_{\mathds{C}}$, et si $\mathcal{M}$ est une connexion sur $\mathds{A}^{n}_{\mathds{C}}$ méromorphe le long de $D$, alors le $\mathcal{D}$-module défini par la construction de \cite{clean} se restreint aux fibres de $T$ en une somme de modules exponentiels associés à des formes linéaires. C'est la propriété de \textit{linéarité} \ref{deflin}. Dans le cas d'un germe formel de courbe lisse sur un corps de caractéristique $0$, le théorème de linéarité \ref{conj} est précisé dans \cite{Tey} par une formule explicite.
\\ \indent
La condition imposée à $\mathcal{M}$ dans \ref{conj} est générique sur $D$, alors que dans \cite{clean} elle concerne tous les points du diviseur. Le théorème \ref{conj} suggère donc que dans le contexte $\ell$-adique, on doit pouvoir obtenir un théorème d'additivité à l'aide d'une condition portant seulement sur le point générique de $D$. Pour les connexions méromorphes, travailler avec cette condition plus faible a cependant le désavantage que l'on obtient a priori pas d'autre information sur les $\mathcal{H}^{i}$ supérieurs des restrictions que leur lissité. 
\\ \indent
En un point de bonne décomposition formelle\footnote{appelé aussi point non tournant dans la littérature.}, les formes linéaires produites par la construction d'Abbes et Saito s'obtiennent comme en dimension 1 à partir des formes différentielles de la décomposition de Levelt-Turrittin générique de $\mathcal{M}$. Le fait surprenant, conséquence de \ref{conj}, est que de telles formes linéaires existent aussi en un point tournant pour $\mathcal{M}$. Ce qui se passe en un tel point ne sera pas abordé dans cet article.
\\ \indent
Pour prouver \ref{conj}, une difficulté est la non-exactitude à gauche de l'image inverse pour les $\mathcal{D}$-modules. Elle fait de la notion de linéarité ponctuelle \ref{linpon} d'un module holonome sur un fibré vectoriel $E$ une notion relativement mauvaise puisque non stable par sous-objet. Pour contourner ce problème, on introduit pour tout $\mathcal{D}$-module sur $E$ une condition \ref{Lxdef} plus forte qui se comporte bien vis-à-vis des sous-objets et de leurs cycles proches. L'analyse de cette interaction est traitée en \ref{grmoins} et \ref{sousmodule}. 
\\ \indent
Suivant le même canevas, on introduit en \ref{Pxdef} une condition aux propriétés similaires impliquant la lissité des $\mathcal{H}^{i}$  des restrictions aux fibres de $E$.
\\ \indent  
La seconde difficulté est la mise en défaut de la décomposition de Levelt-Turrittin en dimension $>1$. Pour y palier, on utilise les réseaux canoniques de Malgrange \cite{Mal96}, qui généralisent aux connexions méromorphes quelconques les réseaux canoniques de Deligne \cite{Del} définis dans le cas à singularité régulière. C'est l'outil  permettant de montrer que la construction d'Abbes et Saito vérifie les conditions \ref{Lxdef} et \ref{Pxdef}. Ce dernier point fait l'objet de \ref{loctri} et \ref{malgrange}\footnote{En dimension 2, les réseaux de Malgrange sont localement libres \cite[3.3.1]{Mal96}. En particulier, \ref{malgrange} est en dimension 2 l'analogue analytique d'un théorème de forme normale d'André \cite[3.3.2]{Andre} pour les connexions méromorphes algébriques formelles, obtenu dans \textit{loc. it.} par des méthodes algébriques.}.
\\ \indent

Je remercie Claude Sabbah pour ses conseils durant  l'élaboration de ce travail. Je remercie Takeshi Saito pour d'innombrables remarques sur une première version de ce texte, ainsi qu'Yves André pour m'avoir fait remarquer qu'en dimension $>2$, les réseaux de Malgrange ne sont pas nécessairement localement libres.

\section{Modules linéaires}
\subsection{Généralités}
\begin{definition}\label{deflin}
Soit $E$ est un $\mathds{C}$-espace vectoriel de dimension finie. On appelle \textit{module linéaire sur $E$} tout module somme directe finie de modules du type $\mathcal{E}^{\varphi}:=(\mathcal{O}_{E},d+d\varphi)$ où $\varphi: E \longrightarrow \mathds{C}$ est une forme linéaire. 
\end{definition}

\begin{lemme}\label{extension}
Soient  $\mathcal{E}_{1}$ et  $\mathcal{E}_{2}$ des modules linéaires sur $\mathds{A}^{n}_{\mathds{C}}$. Alors $$\Ext^{1}_{\mathcal{D}_{\mathds{A}^{n}_{\mathds{C}}}}(\mathcal{E}_{1},\mathcal{E}_{2})\simeq 0.$$
\end{lemme}
\begin{proof}
On peut supposer que les $\mathcal{E}_{i}$ sont de rang 1. Par tensorisation, on peut aussi supposer que $\mathcal{E}_{1}$ est trival. Il faut donc montrer la nullité du premier groupe de cohomologie de De Rham algébrique $\DR\mathcal{E}$ de $\mathcal{E}:=\mathcal{E}_{1}$ avec $\mathcal{E}$ donné par la 1-forme $a_{1}dx_{1}+ \cdots + a_{n}dx_{n}$ où $a_{i}\in \mathds{C}$. Si $n=1$, c'est un calcul immédiat. Supposons donc $n>1$. \\ \indent
Si tous les $a_{i}$ sont nuls, $\DR\mathcal{E}$ est le complexe de Koszul $\mathcal{K}(\partial_{1},\dots ,\partial_{n})$ pour le module $\mathds{C}[x_{1},\dots ,x_{n}]$ sur l'anneau commutatif $\mathds{C}[\partial_{1},\dots ,\partial_{n}]$. Il est en particulier acyclique en degré $>0$. 
Si l'un des $a_{i}$ est non nul, mettons $a_{1}$ on peut supposer quitte à poser $y_{1}=a_{1}x_{1}+ \cdots + a_{n}x_{n}, y_{2}=x_{2}, \dots,  y_{n}=x_{n}$ que $a_{2}=\dots = a_{n}=0$. Dans ce cas, $\DR\mathcal{E}$ est le complexe de Koszul $\mathcal{K}(\partial_{1}+1, \partial_{2}, \dots ,\partial_{n})$ de $\mathds{C}[x_{1},\dots ,x_{n}]$. Or ce dernier complexe est quasi-isomorphe au cône du morphisme de complexe
\[\xymatrixcolsep{3pc}
\xymatrix{ 
\mathcal{K}(\partial_{2}, \dots ,\partial_{n}) \ar[r]^-{(\partial_{1}+1)\cdot}& \mathcal{K}(\partial_{2}, \dots ,\partial_{n})
}
\]
avec $\mathcal{K}(\partial_{2}, \dots ,\partial_{n})$ quasi-isomorphe à $\mathds{C}[x_{1}]$ placé en degré 0, d'où l'annulation voulue.
\end{proof}
\begin{corollaire}\label{déguisement}
Soit $\mathcal{E}$ un fibré vectoriel algébrique à connexion intégrable sur  $\mathds{A}^{n}_{\mathds{C}}$ muni des coordonnées $(x_{1},\dots , x_{n})$. On suppose l'existence d'une trivialisation globale de $\mathcal{E}$ sur laquelle les $\partial_{x_{i}}$ agissent via des matrices à coefficients constants. Alors $\mathcal{E}$  est  linéaire.
\end{corollaire}
\begin{proof}
On raisonne par récurrence sur le rang de $\mathcal{E}$, le cas où $\mathcal{E}$ est de rang 1 étant trivial. Notons $\mathbf{e}$ une base ayant la propriété de l'énoncé. Par intégrabilité de $\mathcal{E}$, les matrices des $\partial_{x_{i}}$ commutent deux à deux, donc une base de triangularisation simultanée donne une suite exacte
$$
0 \longrightarrow \mathcal{E}_{1}\longrightarrow \mathcal{E}\longrightarrow \mathcal{E}_{2}\longrightarrow 0
$$
avec $\mathcal{E}_{1}$ linéaire de rang 1 et $\mathcal{E}_{2}$ satisfaisant aux conditions de l'énoncé. Par récurrence, les $\mathcal{E}_{i}$ sont linéaires et \ref{déguisement} se déduit de \ref{extension}.
\end{proof}
\begin{corollaire}\label{sousquotient}
Tout sous-quotient d'un module linéaire est linéaire.
\end{corollaire}
\begin{proof}
On va montrer \ref{sousquotient} pour les sous-objets, le cas des quotients étant similaire.
Soit $\mathcal{E}\simeq \mathcal{E}\simeq \mathcal{E}^{\varphi_{1}}\oplus \dots \oplus \mathcal{E}^{\varphi_{n}}$ un module linéaire, et 
soit $\mathcal{M}$ un sous-module de $\mathcal{E}$. On veut montrer que $\mathcal{M}$ est linéaire. Du fait de  $\Char(\mathcal{M})\subset \Char(\mathcal{E})=T^{\ast}_{\mathds{A}^{n}_{\mathds{C}}}\mathds{A}^{n}_{\mathds{C}}$, on a $\Char(\mathcal{M})=T^{\ast}_{\mathds{A}^{n}_{\mathds{C}}}\mathds{A}^{n}_{\mathds{C}}$, donc $\mathcal{M}$ est une connexion algébrique. \\ \indent
Raisonnons par récurrence en supposant \ref{sousquotient} acquis pour tous les couples $(\mathcal{M}^{\prime}, \mathcal{E}^{\prime})$ avec $\mathcal{M}^{\prime}$ sous-module de $\mathcal{E}^{\prime}$ satisfaisant à $\rg \mathcal{E}^{\prime}< \rg \mathcal{E}$ ou ($\rg \mathcal{E}^{\prime}=\rg \mathcal{E}$ et  $\rg \mathcal{M}^{\prime}<\rg \mathcal{M}$). \\ \indent 
Si $\mathcal{M}$ est simple et non nul, alors la restriction de l'un des $p_{i}:\mathcal{E} \longrightarrow \mathcal{E}^{\varphi_{i}}$ à $\mathcal{M}$ est un isomorphisme. Sinon, $\mathcal{M}$ admet un sous-objet propre $\mathcal{N}$. Par hypothèse de récurrence appliquée à $(\mathcal{N},\mathcal{E})$, le module $\mathcal{N}$ est linéaire. Soit $\mathcal{N}^{\prime}$ un facteur linéaire de rang $1$ de $\mathcal{N}$. Par hypothèse de récurrence appliquée à $(\mathcal{M}/\mathcal{N}^{\prime},\mathcal{E}/\mathcal{N}^{\prime})$, le module $\mathcal{M}$ est une extension de deux modules linéaires. D'après \ref{extension}, $\mathcal{M}$ est linéaire.
\end{proof}

Soit $E$ un fibré vectoriel sur une variété  algébrique complexe lisse $X$
et soit $x\in X$. Pour tout couple $(Y,Z)$ de sous-variétés de  $X$ tel que $Z\subset Y$, on note $i_{Z,Y}$ l'inclusion canonique de $E_Z$ dans $E_Y$.
Soit $\mathcal{M}$ un $\mathcal{D}$-module holonome sur $E$. 
\begin{definition}\label{linpon}
On dit que $\mathcal{M}$ est $\mathcal{H}^{0}$-\textit{linéaire en $x$} si le module $\mathcal{H}^{0}i^{+}_{x,X}\mathcal{M}$ est linéaire. On dit que $\mathcal{M}$ est $\mathcal{H}^{0}$-\textit{ponctuellement linéaire} si $\mathcal{M}$ est $\mathcal{H}^{0}$-linéaire en tout point de $X$.
\end{definition}
\subsection{La propriété $L$}
Soit $M$ un $\mathcal{D}$-module holonome sur un fibré $E$ de rang $l$ sur $\Spec \mathds{C}\llbracket t_{1}, \dots , t_{n}\rrbracket$, et soit $O$  le point fermé de $\Spec \mathds{C}\llbracket t_{1}, \dots , t_{n}\rrbracket$.
\begin{definition}\label{Lxdef}
On dit que $M$ \textit{vérifie la propriété $L$} si sur un voisinage de $E_O$ dans $E$, le module $M$ admet une famille génératrice $\mathbf{e}:=(e_{1},\dots , e_{m})$ vérifiant
\begin{equation}\label{Lx}
\partial_{y_{i}}e_{j}=\displaystyle{\sum_{u=1}^{m}}f_{iju}(t,y)e_{u}
\end{equation}
où dans cette écriture, les $y_{i}$ sont des coordonnées sur $E$, et les $f_{iju}$ sont de la forme $g_{iju}/h_{iju}$ avec $g_{iju},h_{iju} \in \mathds{C}\llbracket t_{1}, \dots , t_{n}\rrbracket[y_1,\dots ,y_l]$ ayant la propriété\footnote{Appelée par abus \textit{propriété $L$} dans la suite.} que dans la décomposition
$$
g_{iju}=\displaystyle{\sum_{\nu}}P_{\nu}(y_1,\dots ,y_l)t_1^{\nu_1}\cdots t_n^{\nu_n}
$$
le degré total $\deg_y P$ de $P_{\nu}(y_1,\dots ,y_l)$ est plus petit que  $\nu_1+\dots +\nu_n$, et de même pour les $h_{iju}$. \\ \indent
Dans la situation globale de \ref{linpon}, on dira aussi que $\mathcal{M}$ \textit{vérifie la propriété $L$ en $x$} si pour un choix d'identification $\hat{\mathcal{O}_{X,x}}\simeq \mathds{C}\llbracket t_{1}, \dots , t_{n}\rrbracket$, le changement de base de $\mathcal{M}$ à $\Spec \hat{\mathcal{O}_{X,x}}$ satisfait à la propriété $L$ au sens précédent.
\end{definition}

\begin{lemme}\label{Lximpliquelinéaire}
Si $\mathcal{M}$ a la propriété $L$ en $x$, alors $\mathcal{M}$ est $\mathcal{H}^{0}$-linéaire en $x$.
\end{lemme}
\begin{proof}
$\mathcal{H}^{0}i^{+}_{x,X}\mathcal{M}$ est la restriction à $E_x$ de $M=\hat{\mathcal{O}_{X,x}}\otimes_{\mathcal{O}_{X,x}}\mathcal{M}_{x}$. Une famille génératrice de $M$ vérifiant \eqref{Lx} induit une famille génératrice de $\mathcal{H}^{0}i^{+}_{x,X} M$ encore notée $\mathbf{e}$ sur laquelle l'action de $\partial_{y_i}$ s'obtient en évaluant les fonctions $f_{iju}$ en $t_1=\dots=t_n=0$. Par hypothèse, $f_{iju}(0,y)$ est constante. \\ \indent
Soit $\mathbf{e}^{\prime}$ une sous famille maximale de $\mathbf{e}$ n'admettant pas de relations de liaison non triviale à coefficients constants. Une telle famille existe si $M\neq 0$, et quitte à renuméroter les $e_i$, on peut supposer $\mathbf{e}^{\prime}=(e_1,\dots , e_k)$, $k\leq n$. Par maximalité, tous les $e_i$ sont dans le $\mathds{C}$-espace vectoriel engendré par $\mathbf{e}^{\prime}$. Donc $\mathbf{e}^{\prime}$ engendre $\mathcal{H}^{0}i^{+}_{x,X}M$ comme $\mathcal{D}_{E_{x}}$-module. En particulier, les relations
$$
\partial_{y_i}e_j=a_{ij1}e_1+\cdots + a_{ijn}e_n, \quad a_{iju}\in \mathds{C}
$$
pour $j=1,\dots, k$ donnent des relations
\begin{equation}\label{uneequation}
\partial_{y_i}e_j=b_{ij1}e_1+\cdots + b_{ijk}e_k, \quad b_{iju}\in \mathds{C}
\end{equation}
En notant $B_i$ la matrice des $(b_{iju})_{ju}$, les relations \eqref{uneequation} s'écrivent 
$
\partial_{y_i}\mathbf{e}^{\prime}=B_i\mathbf{e}^{\prime}
$.
Par application de $\partial_{y_j}$ et commutation de $\partial_{y_i}$  et $\partial_{y_j}$, on obtient
$$
(B_iB_j-B_jB_i)\mathbf{e}^{\prime}=0
$$
Par définition de $\mathbf{e}^{\prime}$, on a $B_iB_j=B_jB_i$. Alors, la connexion $$N:=(\mathds{C}[y_1,\dots ,y_l]^{k}, B_1dy_1+\cdots+B_ldy_l)$$
est une connexion algébrique intégrable bien définie se surjectant sur $\mathcal{H}^{0}i^{+}_{x,X} M$. d'après \ref{déguisement}, le module $N$ est linéaire. Par \ref{sousquotient}, le module $\mathcal{H}^{0}i^{+}_{x,X} M$ est linéaire.

\end{proof}
La définition \ref{Lxdef} s'introduit naturellement pour palier au mauvais comportement de la linéarité ponctuelle vis-à-vis des sous-objets. 
On démontrera en \ref{sousmodule} par l'entremise de la théorie des cycles proches que tout sous-objet d'un module vérifiant la propriété $L$ en $x$ est $\mathcal{H}^{0}$-linéaire en $x$.

\section{Le théorème principal}
\subsection{Enoncé}
On commence par quelques rappels sur \cite{clean}.  
Soit $X$ une variété algébrique complexe lisse, $D$ un diviseur lisse de $X$ et $U:=X\setminus D$. \\ \indent
Soit $\tilde{X\times X}$ l'éclaté de $X\times X$ le long de $D\times D$. On note $X\ast X$ le complémentaire dans  $\tilde{X\times X}$ des transformées strictes de $D\times X$ et $X\times D$. Si $V_{1}$ et $V_{2}$ sont deux ouverts de $X$ d'anneaux de fonctions respectifs $A_{1}$ et $A_{2}$, et si $t=0$ (resp. $s=0$) est l'équation de $D$ dans $V_{1}$ (resp. $V_{2}$), $X\ast X$ est le schéma d'anneau de fonctions \cite[5.22]{clean}
\begin{equation}\label{ast}
\frac{A_{1}\otimes_{\mathds{C}}A_{2}[u^{\pm1}]}{(t\otimes 1 -u(1\otimes s))}
\end{equation} 
Le schéma $X\ast X$ s'insère dans le diagramme commutatif
\begin{equation*} \label{carrecartesien}
\begin{array}{c}
\xymatrix{ 
                          &           X\ast X \ar[d] \\  
            X       \ar[r]_-{\text{Diag}} \ar[ru]^-{\delta}     &         X\times X           
}
\end{array}
\end{equation*}
Avec $\delta$ immersion fermée régulière de fibré conormal canoniquement isomorphe à $\Omega^{1}_{X}(\log D)$. 
Si $a\in \mathds{N}^{\ast}$ on s'intéresse comme en dimension 1 \cite[3]{AS} à la dilatation $(X\ast X)^{a}$ de $X\ast X$ en $\delta(aD)$ par rapport à $aD$. On a un diagramme commutatif de carré gauche cartésien
\begin{equation} \label{carrecartesien}
\begin{array}{c}
\xymatrix{ 
   T_{a} \ar[d]\ar[r]  &     (X\ast X)^{a}        \ar[d]_-{\pi}        &   \ar[l]_-{j_a}       U\times U \ar[d]\\  
      D             \ar[r]   &                 X                            &         U    \ar[l]       
}
\end{array}
\end{equation}
avec  \cite[(5.31.4)]{clean}
\begin{equation*} 
\begin{array}{c}
\xymatrix{ 
T_{\underline{a}} \ar[r]^-{\sim} &\mathbf{V}(\Omega^{1}_{X/\mathds{C}}(\log D)\otimes_{\mathcal{O}_{X}} \mathcal{O}_{X}(aD))
}\times_{X}D.
\end{array}
\end{equation*}
Dans la suite, on se restreint au cas où $X$ est l'espace affine $\mathds{A}^{n}_\mathds{C}$ avec $D$ un hyperplan, et on se donne une connexion méromorphe $\mathcal{M}$  sur  $\mathds{A}^{n}_\mathds{C}$ à pôles le long de $D$. On notera $\rho(\mathcal{M})$ où encore $\rho$ le rang de Poincaré-Katz de $\mathcal{M}$ au point générique de $D$ (c'est-à-dire la plus grande pente de la restriction de $\mathcal{M}$ au point générique de $D$),  et $H_{a}(\mathcal{M})$ pour $j_{a+}\mathcal{H}om(p^{+}_{2}\mathcal{M}, p^{+}_{1}\mathcal{M})$, où $j_{a}$ désigne l'inclusion canonique de $U\times U$ dans  $(X\ast X)^{a}$. \\ \indent
On dira qu'un module holonome $\mathcal{N}$ sur un fibré sur $D$ est \textit{ponctuellement lisse} si les modules de cohomologie du complexe $i^{+}_{x,X}\mathcal{N}$ sont des connexions algébriques pour tout $x\in D$. Notons $\psi_{\pi}$ le foncteur des cycles proches par rapport à $\pi$. Le but de cet article est de prouver le
\begin{theorem}\label{conj}
On suppose que $\rho(\mathcal{M})$ est entier non nul. Alors, si $a\geq \rho(\mathcal{M})$, le $\mathcal{D}_{T_a}$-module $\psi_{\pi} H_{a}(\mathcal{M})$ est ponctuellement lisse et ponctuellement $\mathcal{H}^{0}$-linéaire. 
\end{theorem}

\subsection{}
Abbes et Saito utilisent le foncteur de restriction plutôt que les cycles proches. Le choix des cycles proches se justifie ici par le fait
que la restriction à une sous-variété $Z$ de n'importe quel $\mathcal{D}$-module localisé le long de $Z$ donne $0$, alors que $\psi$ est insensible à la localisation \cite[4.4-3]{MM}. 
\subsection{}
Dans \cite[2.3.7]{Saito} et \cite[8.15]{clean}, l'analogue   $\ell$-adique de \ref{conj} est démontré à l'aide d'une hypothèse \cite[8.2]{clean} sur la ramification le long de $D$. Soit $X$ une variété sur un corps parfait $k$ de caractéristique $p>0$, et $D$ un diviseur lisse de $X$. On considère le diagramme cartésien
\begin{equation} \label{dcartesien}
\begin{array}{c}
\xymatrix{ 
          U      \ar[d]\ar[r]^-{\delta}                    &           U\times_{k} U \ar[d]^{j_{a}} \\  
            X      \ar[r]_-{\delta_{a}}  &                  (X\ast X)^{a}           
}
\end{array}
\end{equation}
Soit $\ell\neq p$ un nombre premier,  $x \in D$,  $\overline{x}$ un point géométrique de $X$ localisé en $x$, et $\mathcal{F}$ un $\mathds{F}_{\ell}$-faisceau localement constant constructible sur $U$. Alors, le changement de base relatif à \eqref{dcartesien}
$$
\alpha : \delta_{a}^{\ast}j_{a\ast}H_{a}(\mathcal{F})\longrightarrow j_{\ast}\delta^{\ast}H_{a}(\mathcal{F})=j_{\ast}\End(\mathcal{F})
$$
est injectif, et on dit que \textit{la ramification de $\mathcal{F}$ en $\overline{x}$ est bornée par $aD$} si  $\alpha_{\overline{x}}$ est un isomorphisme. Sous cette condition, l'objet d'étude d'Abbes et Saito est le faisceau $(j_{a\ast}\mathcal{H}om(p^{\ast}_{2}\mathcal{F}, p^{\ast}_{1}\mathcal{F}))_{|T_{a}}$. \\ \indent
Pour les connexions méromorphes, $\alpha$ est  un isomorphisme \cite[1.7.3]{HTT}, et il se pose alors la question de savoir par quoi remplacer la condition d'Abbes et Saito. Si on choisit pour $x$ le point générique $\eta$ de $D$, elle équivaut au fait que la plus grande pente de la restriction de $\mathcal{F}$ au point générique de $\Spec \mathcal{O}_{X,\eta}^{sh}$ est $\leq a$. Voir \cite[8.8]{clean}. Cette condition fait sens pour les $\mathcal{D}$-modules, d'où l'occurrence du rang de Katz générique dans \ref{conj}. \\ \indent
Le théorème \ref{conj} suggère que dans le contexte $\ell$-adique, on doit pouvoir obtenir un théorème d'additivité à l'aide d'une condition portant seulement sur le point générique de $D$.

\section{Cycles proches}
\subsection{Rappels et notations}
Dans cette section, on fixe une variété complexe lisse $X$, une hypersurface lisse $D$ de $X$ définie par une équation $t=0$  et $\mathcal{M}$ un $\mathcal{D}_{X}$-module holonome. On rappelle suivant \cite{MM} que $\mathcal{D}_{X}$ est muni de la \textit{filtration de Kashiwara-Malgrange}
$$
V_{k}(\mathcal{D}_{X}):=\{P \in \mathcal{D}_{X},  P(\mathcal{I}_{D}^{l})\subset \mathcal{I}_{D}^{l-k} \quad \forall l \in \mathds{Z}\}
$$
et qu'on appelle \textit{bonne $V$-filtration} de $\mathcal{M}$ toute filtration exhaustive $(U_{k}(\mathcal{M}))_{k\in \mathds{Z}}$ par des $V_{0}(\mathcal{D}_{X})$-modules cohérents tels que localement, il existe un entier $k_{0}\in \mathds{N}$ vérifiant pour tout $k\in \mathds{N}$
\begin{equation}\label{bonnefiltration}
U_{-k-k_{0}}(\mathcal{M})=t^{k}U_{-k_{0}}(\mathcal{M}) \quad \text{ et } \quad U_{k+k_{0}}(\mathcal{M})= \displaystyle{\sum_{i=0}^{k}}\partial_{t}^{i}U_{k_{0}}(\mathcal{M})
\end{equation}
Si $m$ est une section de $\mathcal{M}$, il existe un polynôme de degré minimal $b_{m}$ pour lequel on a 
$$
b_{m}(t\partial_{t})m\in V_{-1}(\mathcal{D}_{X})m.
$$
C'est le \textit{polynôme de Bernstein} de $m$. Notons $\ord_{Y}(m)$ l'ensemble de ses racines et soit $\geq$ l'ordre lexicographique sur $\mathds{C}\simeq \mathds{R}+ i\mathds{R}$. Pour $a \in \mathds{C}$, si on définit 
$$
V_{a}(\mathcal{M})=\{m \in \mathcal{M},  \ord_{Y}(m) \subset \{ \alpha \in \mathds{C}, \alpha \geq -a-1 \}\}
$$
et 
$$
V_{<a}(\mathcal{M})=\{m \in \mathcal{M},  \ord_{Y}(m) \subset \{ \alpha \in \mathds{C}, \alpha > -a-1 \}\},
$$
Alors $(V_{a+k}(\mathcal{M}))_{k\in \mathds{Z}}$ et $(V_{<a+k}(\mathcal{M}))_{k\in \mathds{Z}}$ sont des bonnes $V$-filtrations de $\mathcal{M}$, et ce sont les seules bonnes $V$-filtrations de $\mathcal{M}$ dont les racines du polynôme de Bernstein sont respectivement dans $[-a-1,-a[$ et $]-a-1,-a]$. Si on pose $\Gr_{a}(\mathcal{M})=V_{a}(\mathcal{M})/V_{<a}(\mathcal{M})$, on a par définition
$$
\Psi_{t}\mathcal{M}:= \displaystyle{\bigoplus_{-1 \leq a<0}}\Gr_{a}(\mathcal{M}).
$$

\subsection{Généralités sur les $V$-filtrations}
\begin{lemme}\label{Vfiltration}
Suposons que $\mathcal{M}$ est localisé le long de $D$, à savoir $\mathcal{M}\simeq \mathcal{M}[t^{-1}]$. Soit $U_{0}(\mathcal{M})$ un $V_{0}(\mathcal{D}_{X})$ sous-module cohérent de $\mathcal{M}$ tel que $\mathcal{M}=U_{0}(\mathcal{M})[t^{-1}]$. Alors, la filtration de $\mathcal{M}$ définie par 
$$
U_{k}(\mathcal{M})=t^{-k}U_{0}(\mathcal{M}) \quad (k \in \mathds{Z})
$$
est une bonne $V$-filtration de $\mathcal{M}$.
\end{lemme}
\begin{proof}
Du fait de $(t\partial_{t})t^{-k}=-kt^{-k}+t^{-k}(t\partial_{t})$, chaque $U_{k}(\mathcal{M})$ est un $V_{0}(\mathcal{D}_{X})$-module cohérent. Il faut donc vérifier que les relations de \eqref{bonnefiltration} sont satisfaites pour un certain $k_{0}\geq 0$. Par définition, la première relation de \eqref{bonnefiltration} est vérifiée pour tout choix de $k_{0}$. \\ \indent
Montrer la seconde relation revient à montrer
\begin{equation}\label{bonneV}
U_{k+k_{0}}(\mathcal{M})=\displaystyle{t^{k}\sum_{i=0}^{k}}\partial_{t}^{i}U_{k_{0}}(\mathcal{M})
\end{equation}
pour un $k_{0}$ convenable. Du fait de $\partial_{t}t^{-k}=-kt^{-k-1}+t^{-k-1}t\partial_{t}$, la stabilité de $U_{0}(\mathcal{M})$ par $t\partial_{t}$ entraine $\partial_{t} U_{k}(\mathcal{M}) \subset U_{k+1}(\mathcal{M})$  pour tout $k\in \mathds{Z}$. L'inclusion
$\supset$ dans \eqref{bonneV} est donc automatique. Il faut montrer l'inclusion $\subset$ pour un $k_{0}$ bien choisi. \\ \indent
Si $e_{1},\dots , e_{n}$ est un système de $V_{0}(\mathcal{D}_{X})$-générateurs de $U_{0}(\mathcal{M})$, on observe que le polynôme $b(T)=b_{e_{1}}(T)\cdots b_{e_{n}}(T)$ annule l'opérateur induit par $t\partial_{t}$ sur $\Gr_{0}U(\mathcal{M})= U_{0}(\mathcal{M})/U_{-1}(\mathcal{M})$. Par conséquent, le polynôme $b(T+k)$ annule l'opérateur induit par $t\partial_{t}$ sur $\Gr_{k}U(\mathcal{M})= U_{k}(\mathcal{M})/U_{k-1}(\mathcal{M})$ pour tout $k\in \mathds{Z}$. Donc si $A$ est l'ensemble fini des valeurs propres de l'action de $t\partial_{t}$ sur $\Gr_{0}U(\mathcal{M})$, les valeurs propres de l'action de $t\partial_{t}$ sur $\Gr_{k}U(\mathcal{M})$ sont dans $A-k$. Choisissons $k_{0}$ assez grand  tel que $A-k_{0}$ ne rencontre par $\mathds{N}$, et montrons  
\eqref{bonneV} pour $k=k_{0}$. On raisonne par récurrence sur $k$, le cas $k=0$ étant tautologique. Par récurrence, il suffit de montrer 
$$
U_{k_{0}}(\mathcal{M})\subset U_{k_{0}-1}(\mathcal{M})+t^{k}\partial_{t}^{k}U_{k_{0}}(\mathcal{M}),
$$
soit encore que le morphisme $T_{k}: \Gr_{k_{0}}U(\mathcal{M}) \longrightarrow \Gr_{k_{0}}U(\mathcal{M})$ induit par l'opérateur $t^{k}\partial_{t}^{k}$ est surjectif. Or, une récurrence permet de voir que 
$$
t^{k}\partial_{t}^{k}=t\partial_{t}(t\partial_{t}-1)\cdots (t\partial_{t}-(k-1)), 
$$
de sorte que $T_{k}$ admet un polynôme minimal non nul $\mu_{T_{k}}=X^{d}+a_{d-1}X^{d-1}+\dots +a_{0}$ dont les racines sont de la forme $\lambda (\lambda-1)\cdots (\lambda-(k-1))$ pour $\lambda \in A-k_{0}$. Ces racines sont donc non nulles, soit encore $a_{0}\neq 0$. On a donc pour tout $m \in U_{k_{0}}(\mathcal{M})$ l'égalité suivante dans $\Gr_{k_{0}}U(\mathcal{M})$
$$
[m]=-(T_{k}^{d}[m]+a_{d-1}T_{k}^{d-1}[m]+\dots +a_{1}T_{k}[m])/a_{0}
$$
et la surjectivité souhaitée est prouvée.
\end{proof}
On aura besoin du lemme \cite[4.2-1]{MM}
\begin{lemme}\label{2inclusions}
Soient $U$ et $U^{\prime}$ deux bonnes $V$-filtrations de $\mathcal{M}$. Alors, il existe localement des entiers $k_{1}, k_{2}\in \mathds{Z}$ tels que
$$
U_{k+k_{1}} \subset U^{\prime}_{k}  \subset U_{k+k_{2}} \quad \forall k \in \mathds{Z}
$$
\end{lemme}

\subsection{Cycles proches et propriété $L$}\label{cpetLx}
Dans le lemme qui suit, $E$ désigne le fibré trivial de rang $m$ sur $X=\Spec\mathds{C}\llbracket t_{1}, \dots , t_{n}\rrbracket$, muni de coordonnées  $y_1,\dots ,y_m$, $O$ le point fermé de $X$ et $M$ désigne un module holonome sur $E$. Soit $D$ défini par $t_n=0$ et $E_D$ la restriction de $E$ à $D$. On munit $M$ de la $V$-filtration associée à $E_D$.
\begin{lemme}\label{grmoins}
Si $M$ vérifie la propriété $L$ et si $N$ est un sous-module de $M$, alors pour tout nombre complexe $a<0$, le gradué $\Gr_{a}(N)$ est un sous-module d'un module vérifiant $L$.
\end{lemme}
\begin{proof}
Par exactitude des foncteurs $\Gr_{a}$, on peut supposer que $N=M$ \cite[4.2-7]{MM}. Si $\mathbf{e}=(e_{1},\dots ,e_{m})$ est une famille génératrice de $M$ donnée par propriété $L$ de $M$, la suite des modules\footnote{On note de façon abusive $e_j$ pour l'image de $e_j$ par $M\longrightarrow M[t^{-1}_n]$. }  $M_{k}:=\mathcal{D}_{E}t^{-k}_n\mathbf{e}, k\in \mathds{N}$ est une suite croissante exhaustive de sous-modules de $M[t^{-1}_n]$. Par noethérianité de $M[t^{-1}_n]$ \cite[V 1.9]{Borel}, $(M_{k})_{k\in \mathds{N}}$ stationne à partir d'un certain entier $k_{0}$. Alors, la famille des $t^{-k_{0}}_n e_{j}$ fait de $M[t^{-1}_n]$ un module vérifiant la propriété $L$. Puisque par \cite[4.4-3]{MM} on a $\Gr_{a}(M)=\Gr_{a}(M[t^{-1}_n])$, on peut donc supposer $M=M[t^{-1}_n]$. \\ \indent
Soit $\mathbf{e}:=(e_{1},\dots , e_{m})$ une famille génératrice de $M$ vérifiant \eqref{Lx} et notons $V(\mathbf{e})$ le $V_{0}(\mathcal{D}_{E})$-module engendré par $\mathbf{e}$. Si $s\in M$, on peut écrire
\begin{align*}
s& =\sum P_{i}(t_{j},y_{j},\partial_{t_{1}},\dots ,\partial_{t_{n-1}}, \partial_{t_n}, \partial_{y_{j}})e_{i} \\
 & =\sum P_{i}(t_{j},y_{j},\partial_{t_{1}},\dots ,\partial_{t_{n-1}},t^{-1}_n(t_n\partial_{t_n}), \partial_{y_{j}})e_{i}
\end{align*}
où les $P_i$ désignent des polynômes à plusieurs variables et à coefficients complexes. On en déduit que $V(\mathbf{e})[t^{-1}_n]=M$, et alors par \ref{Vfiltration}
$$
V_{k}(\mathbf{e}):=t^{-k}_nV(\mathbf{e})
$$ 
est une bonne $V$-filtration de $M$. \\ \indent
D'après \ref{2inclusions}, on peut trouver $k_{0}\in  \mathds{N}$ tel que
$$
V_{-k_{0}}(\mathbf{e})\subset V_{<a}(M) \subset V_{a}(M) \subset V_{k_{0}}(\mathbf{e})
$$
d'où une surjection canonique de  $V_{0}(\mathcal{D}_{E})$-modules
\begin{equation}\label{surj}
\xymatrix{
V_{k_{0}}(\mathbf{e})/V_{-k_{0}}(\mathbf{e})\ar@{->>}[r]
& V_{k_{0}}(\mathbf{e})/ V_{<a}(M)
}
\end{equation}
et une injection canonique de $V_{0}(\mathcal{D}_{E})$-modules
\begin{equation}\label{Grincluseipe}
\xymatrix{
\Gr_{a}(M)\ar@{^{(}->}[r]
& V_{k_{0}}(\mathbf{e})/ V_{<a}(M)
}
\end{equation}
Pour prouver \ref{grmoins}, il suffit d'après \eqref{Grincluseipe}  de voir que par restriction des scalaires de $V_{0}(\mathcal{D}_{E})$ à $\mathcal{D}_{E_D}$, le $\mathcal{D}_{E_D}$-module induit par $V_{k_{0}}(\mathbf{e})/ V_{<a}(M)$ vérifie la propriété $L$. 
Par \eqref{surj}, il suffit de voir que le $\mathcal{D}_{E_D}$-module induit par
$$
V(\mathbf{e},k_{0}):=V_{k_{0}}(\mathbf{e})/V_{-k_{0}}(\mathbf{e})
$$
vérifie la propriété $L$. \\ \indent
Puisque pour tout $k \in \mathds{Z}$ et tout $\alpha\in \mathds{N}$, l'opérateur $t^{k}_n(t_n\partial_{t_n})^{\alpha}$ s'écrit comme combinaison linéaire à coefficients entiers des $(t_n\partial_{t_n})^{i}t^{k}_n, i=0, \dots ,\alpha$, les classes $[t^{k}_ne_{i}], k = -k_{0}, \dots , k_{0} $ forment une famille $V_{0}(\mathcal{D}_{E})$-génératrice finie de $V(\mathbf{e},k_{0})$. Puisque l'action de $t_n\partial_{t_n}$ sur $V(\mathbf{e},k_{0})$ admet un polynôme minimal non nul, disons de degré $d>0$, la famille des $[(t_n\partial_{t_n})^{\alpha}t^{k}_ne_{i}]$ pour $k = -k_{0}, \dots , k_{0}$ et $\alpha = 0, \dots , d-1$ engendre $V(\mathbf{e},k_{0})$ comme $\mathcal{D}_{E_D}$-module. \\ \indent
Montrons que les $[(t_n\partial_{t_n})^{\alpha}t^{k}_ne_{i}]$ satisfont à une relation du type \eqref{Lx}. On écrit
$$
\partial_{y_{i}}e_{j}=\displaystyle{\sum_{u=1}^{m}}f_{iju}(t, y)e_{u}
$$
avec les $f_{iju}$ comme en \ref{Lxdef}. On a
$$
\partial_{y_i}[(t_n\partial_{t_n})^{\alpha}t^{k}_ne_{i}] =
[(t_n\partial_{t_n})^{\alpha}t^{k}_n\partial_{y_a}e_{i}] 
 = \displaystyle{\sum_{u=1}^{m}} [(t_n\partial_{t_n})^{\alpha}t^{k}_nf_{iju}(t, y)e_{u}]
$$
et on est ramené à voir que si $f(t,y)$ est quotient de deux séries vérifiant la propriété $L$, alors il en est de même de $t_n\partial_{t_n}f(t, y)$. Cela découle du fait que le sous-espace de $\mathds{C}\llbracket t_{1}, \dots , t_{n}\rrbracket[y_1,\dots ,y_l]$ des fonctions satisfaisant à la propriété $L$ est une $\mathds{C}$-algèbre stable par $t_n\partial_{t_n}$. Le lemme  \ref{grmoins} est donc acquis.
\end{proof} 

\begin{corollaire}\label{sousmodule}
Tout sous-module $N$ d'un $\mathcal{D}_{E}$-module holonome $M$ vérifiant la propriété $L$ est $\mathcal{H}^{0}$-linéaire.
\end{corollaire}
\begin{proof}
On raisonne par récurrence sur $n$.
Le cas $n=0$ découle de \ref{sousquotient} et \ref{Lximpliquelinéaire}. 
Supposons $n>0$. On dispose d'après \ref{grmoins} d'un hyperplan $i_{D,X}:D \hookrightarrow X$ tel que $\Gr_{-1}(N)$ est inclus dans un module vérifiant $L$. Par hypothèse de récurrence, $\Gr_{-1}(N)$ est $\mathcal{H}^{0}$-linéaire. \\ \indent
Or on sait par \cite[4.4-4]{MM} que $\mathcal{H}^{0}i^{+}_{D,X}N$ est un quotient de $\Gr_{-1}(N)$. Par exactitude à droite de $i^{+}_{O,D}$, on en déduit que $\mathcal{H}^{0}i^{+}_{O,D}\mathcal{H}^{0}i^{+}_{D,X}N\simeq \mathcal{H}^{0}i^{+}_{O,X}N$ est un quotient de $\mathcal{H}^{0}i^{+}_{O,D}\Gr_{-1}(N)$ et on conclut à l'aide de \ref{sousquotient}.

\end{proof} 

\subsection{Cycles proches et propriété $P(x)$}\label{cpetPx}
On adopte les notations de \ref{cpetLx}.
\begin{definition}\label{Pxdef}
On dira que $M$ vérifie la \textit{propriété $P$} si sur un voisinage de $E_O$ dans $E$, le module $M$ admet une famille génératrice $\mathbf{e}:=(e_{1},\dots , e_{m})$ vérifiant
\begin{equation}\label{Px}
\partial_{y_{i}}e_{j}=\displaystyle{\sum_{u=1}^{m}}f_{iju}(t, y)e_{u}
\end{equation}
où les $f_{iju}$ sont des fonctions définies sur un voisinage de $E_O$ dans $E$.
\end{definition}
Bien sûr, si $M$  vérifie la propriété $L$, alors $M$ vérifie aussi la propriété $P$. On a comme en \ref{cpetLx} le 
\begin{lemme}\label{encoregrmoins}
Soit $N$ un sous-module d'un module $M$ vérifiant $P$. Alors pour tout nombre complexe $a$, le gradué $\Gr_{a}(N)$ est inclus dans un module vérifiant $P$.
\end{lemme}
\begin{proof}
Par exactitude de $\Gr_{a}$, on peut supposer $N=M$. Soit $\mathbf{e}=(e_{1}, \dots , e_{n})$ une famille génératrice locale de $\mathcal{M}$ donnée par \ref{Pxdef}. On définit une bonne $V$-filtration par
$$
V_{k}(\mathbf{e}):=V_{k}(\mathcal{D}_{E})\mathbf{e}
$$ 
D'après \ref{2inclusions}, on peut trouver un entier $k_{0}\in  \mathds{N}$ tel que
$$
V_{-k_{0}}(\mathbf{e})\subset V_{<a}(M) \subset V_{a}(M) \subset V_{k_{0}}(\mathbf{e})
$$
d'où une surjection canonique de  $V_{0}(\mathcal{D}_{E})$-modules
$$
\xymatrix{
V_{k_{0}}(\mathbf{e})/V_{-k_{0}}(\mathbf{e})\ar@{->>}[r]
& V_{k_{0}}(\mathbf{e})/ V_{<a}(M)
}
$$
et une injection canonique de $V_{0}(\mathcal{D}_{E})$-modules
$$
\xymatrix{
\Gr_{a}(M)\ar@{^{(}->}[r]
& V_{k_{0}}(\mathbf{e})/ V_{<a}(M)
}
$$
et on obtient \ref{encoregrmoins} en considérant comme en \ref{grmoins} les $\partial_{t}^{i}(t\partial_{t})^{j}e_{k}$ et les $t^{i^{\prime}}(t\partial_{t})^{j}e_{k}$ pour $i,i^{\prime}=0,\dots ,k_{0}$ et $j$ assez grand. 
\end{proof}
\begin{corollaire}\label{encoresousmodule}
Tout sous-module $N$ d'un module vérifiant la propriété $P$ est ponctuellement lisse en $O$.
\end{corollaire}
\begin{proof}
On raisonne par récurrence sur $n$. Le cas où $n=0$ résulte de ce qu'un $\mathcal{D}_{\mathds{A}^{l}_{\mathds{C}}}$-module de type fini sur $\mathds{C}[y_1,\dots ,y_l]$ est une connexion algébrique. Supposons $n>0$ et soit $a\in \mathds{C}$. On dispose d'après \ref{encoregrmoins} d'un hyperplan $i_{D,X}:D \hookrightarrow X$ tel que  $\Gr_{a}(N)$ est inclus dans un module vérifiant $P$. Par hypothèse de récurrence, $\Gr_{a}(N)$  est ponctuellement lisse en $O$. \\ \indent
Or par \cite[4.4-4]{MM}, le complexe $i_{D,X}^{+}N$ est quasi-isomorphe au complexe
\[
\xymatrix{
\Gr_{0}(N)    \ar[r]^-{t_n}
&  \Gr_{-1}(N)
}
\]
On en déduit un isomorphisme de la catégorie dérivée de $\mathcal{D}_{E_{O}}$-mod $$i^{+}_{O,X}(N)\simeq i^{+}_{O,D}(\Gr_{0}(N) \longrightarrow \Gr_{-1}(N))$$ et on conclut à l'aide de la suite spectrale d'hypercohomologie
$$
E_{1}^{pq}=\mathcal{H}^{p}i^{+}_{O,D}(\Gr_{q}(N)) \Longrightarrow  \mathcal{H}^{p+q}i^{+}_{O,X}(N)\quad p\leq 0 \text{ \; et \;} q=0,-1
$$
\end{proof}
\section{Preuve du théorème principal}\label{preuvedeconj}
\subsection{Prologue géométrique}
Soit $(x_{1}, \dots ,x_{n})$ un système de coordonnées  de l'espace affine $\mathds{A}^{n}_{\mathds{C}}$ avec $D$ défini par $x_{n}=0$. On pose $U= \mathds{A}^{n}_{\mathds{C}}\setminus D$. Alors, on a
$$
\mathds{A}^{n}_{\mathds{C}} \ast \mathds{A}^{n}_{\mathds{C}}= \Spec \mathds{C}[x_{i}, t_{i},u^{\pm}]/(x_{n}-ut_{n})
$$
et 
\begin{equation}\label{coordonnees}
(\mathds{A}^{n}_{\mathds{C}} \ast \mathds{A}^{n}_{\mathds{C}})^{(a)}=\Spec \frac{\mathds{C}[x_{i}, t_{i},u^{\pm},y_{1},\dots ,y_{n}]}{(x_{n}-ut_{n}, u-1-t_{n}^{a}y_{n}, (x_{k}-t_{k}-t_{n}^{a}y_{k})_{k<n})}
\end{equation}
Le choix des coordonnées $(x_{1}, \dots ,x_{n})$ induit une identification 
$$
(\mathds{A}^{n}_{\mathds{C}} \ast \mathds{A}^{n}_{\mathds{C}})^{(a)} \simeq \Spec \mathds{C}[t_{1},\dots , t_{n}, y_{1},\dots , y_{n}]
$$
Ce choix étant fait, les coordonnées verticales de $T_{a}$ sont $y_{1}, \dots , y_{n}$, et la projection $(\mathds{A}^{n}_{\mathds{C}} \ast \mathds{A}^{n}_{\mathds{C}})^{(a)} \longrightarrow \mathds{A}^{n}_{\mathds{C}} $ du diagramme \eqref{carrecartesien} est donnée par $(t,y)\longrightarrow t$. Quant à la première projection $U\times U \longrightarrow U$, elle est donnée par
\begin{equation}\label{formulepremiereproj}
(t,y)\longrightarrow (t_1+y_1t_n^{a},\dots , t_{n-1}+y_{n-1}t_n^{a},  t_{n}+y_{n}t_n^{a+1})
\end{equation}

\subsection{$H_a(\mathcal{M})$ vérifie la propriété $L$}
On rappelle que $\rho$ désigne le rang de Poincaré-Katz générique de $\mathcal{M}$.   Posons $\delta_{in}=0$ si $i\neq n$ et  $\delta_{in}=1$ sinon. 
\begin{lemme}\label{loctri}
Si $\mathcal{M}$ est localement engendré comme $\mathcal{D}$-module par un sous-faisceau $\mathcal{O}_{\mathds{A}^{n}_{\mathds{C}}}$-cohérent stable par les $x_i^{\rho+\delta_{in}}\partial_{x_i}$, $i=1,\dots , n$,  alors $H_{a}(\mathcal{M})$ vérifie la propriété $L$ en tout $x \in D$.
\end{lemme}
\begin{proof}
On peut supposer que $x$ est l'origine $O$ de $\mathds{A}_{\mathds{C}}^n$. Notons $\mathcal{R}$ un sous-faisceau $\mathcal{D}_{\mathds{A}^{n}_{\mathds{C}}}$-cohérent de $\mathcal{M}$ comme dans  \ref{loctri}. On a 
$$
H_{a}(\mathcal{M})\simeq p_1^+\mathcal{M}\otimes (p_2^+\mathcal{M})^{\ast}\simeq p_1^+\mathcal{M} \otimes p_2^+(\mathcal{M}^{\ast})
$$
Soit $m$ (resp. $e$) une section de $\mathcal{R}$ (resp. de $\mathcal{M}^{\ast}$) définie au-dessus d'un voisinage de $O$. On note $p_1^+m$ et $p_2^+e$ les sections induites sur $p_1^+\mathcal{M}$ et  $p_2^+(\mathcal{M}^{\ast})$ respectivement.  Par construction,   $\partial_{y_{i}}p_2^+e=0$.  D'après \eqref{formulepremiereproj}, il vient
\begin{align*}
\partial_{y_{i}}  (p_1^+m\otimes p_2^+e)&=(\partial_{y_{i}} p_1^+m)\otimes p_2^+e\\
         & =\frac{\partial}{\partial y_{i}}(t_i+y_it_n^{a+\delta_{in}}) p_1^+\partial_{x_{i}}m    \otimes p_2^+e\\
         &=t_n^{a+\delta_{in}}p_1^+\partial_{x_{i}}m    \otimes p_2^+e \\ 
         &=\frac{(t_n+y_nt_n^{a+1})^{a+\delta_{in}}}{(1+y_nt_n^{a})^{a+\delta_{in}}}p_1^+\partial_{x_{i}}m    \otimes p_2^+e \\ 
         &=\frac{1}{(1+y_nt_n^{a})^{a+\delta_{in}}}p_1^+x_n^{a+\delta_{in}}\partial_{x_{i}}m    \otimes p_2^+e 
\end{align*}
Soit $\mathbf{m}=(m_1, \dots ,m_k)$ une famille $\mathcal{O}_{\mathds{A}^{n}_{\mathds{C}}}$-génératrice locale de $\mathcal{R}$ définie au voisinage de $O$, et $\mathbf{e}=(e_1, \dots ,e_l)$ une trivialisation locale de $\mathcal{M}^{\ast}$ au-dessus du même voisinage. Pour $k\in\mathds{N}$, notons $H_a(\mathcal{M})_k$ le sous-module de $H_a(\mathcal{M})$ engendré par les $p_1^+m_i\otimes (p_2^+e_j)/t^{k}_n$. La suite des $H_a(\mathcal{M})_k$ est croissante. Montrons qu'elle est aussi exhaustive. \\ \indent
On sait déjà que  $\mathcal{N}:=\displaystyle{\bigcup_{k\in\mathds{N}} }H_a(\mathcal{M})_k$ contient $p_1^+\mathcal{R}\otimes p_2^+(\mathcal{M}^{\ast})$. Pour $e\in \mathcal{M}^{\ast}$, $m\in \mathcal{R}$ et $i=1,\dots , n$, on a d'après le calcul précédent
$$
p_1^+\partial_{x_i} m\otimes p_2^+e =\partial_{y_{i}}  (p_1^+m\otimes p_2^+\frac{e}{t_n^{a+\delta_{in}}})\in \mathcal{N}
$$
donc en raisonnant par récurrence sur l'ordre des opérateurs différentiels, il vient
$$
p_1^+Pm \otimes p_2^+(\mathcal{M}^{\ast})\subset \mathcal{N}
$$
pour tout opérateur différentiel $P$. L'exhaustivité de la suite des $H_a(\mathcal{M})_k$ provient alors du fait que $\mathcal{R}$ engendre $\mathcal{M}$ comme $\mathcal{D}$-module. \\ \indent
 Par argument de noethérianité, on a  $H_a(\mathcal{M})=H_a(\mathcal{M})_{k_0}$ pour  $k_0$ assez grand, donc la famille des $p_1^+m_i\otimes (p_2^+e_j)/t^{k_0}_n$ engendre $H_a(\mathcal{M})$ localement. \\ \indent
Puisque $\rho \leq a$, le faisceau $\mathcal{R}$ est stable par $x_n^{ a+\delta_{ln}}\partial_{x_l}$, d'où des relations
$$
x_n^{ a+\delta_{ln}}\partial_{x_l}m_i=f_{li1}(x)m_1+\cdots +f_{lik}(x)m_k
$$ 
pour $l=1,\dots , n$ et $i=1,\dots , k$, avec $f_{liu}\in \mathcal{O}_{\mathds{A}^n_{\mathds{C}},O}$. D'après \eqref{formulepremiereproj}, on a
\begin{align*}
\partial_{y_{l}}  (p_1^+m_i\otimes p_2^+e_j)/t^{k_0}_n
                       &=\frac{1}{(1+y_nt_n^{a})^{a+\delta_{in}}}p_1^+x_n^{a+\delta_{ln}}\partial_{x_{l}}m_i    \otimes  (p_2^+e_j)/t^{k_0}_n  \\
                       &= \frac{1}{(1+y_nt_n^{a})^{a+\delta_{ln}}}    \displaystyle{\sum_{u=1}^{k}}p_1^+f_{liu}(x_1,\dots, x_n)m_u \otimes  (p_2^+e_j)/t^{k_0}_n   \\ 
                       &=  \displaystyle{\sum_{u=1}^{k}}\frac{f_{liu}(t_1+y_1t_n^{a},\dots, t_n+y_nt_n^{a+1})}{(1+y_nt_n^{a})^{a+\delta_{ln}}}   (p_1^+m_u \otimes  (p_2^+e_j)/t^{k_0}_n )
\end{align*}
Puisque $a\geq 1$, la restriction à $\Spec \hat{\mathcal{O}_{\mathds{A}^{n}_{\mathds{C}},O} }$ de la famille des $(p_1^+m_i\otimes p_2^+e_j)/t^{k_0}_n$ fait de $H_a(\mathcal{M})$ un module satisfaisant la propriété $L$ en $O$, et le lemme \ref{loctri} est prouvé.
\end{proof}
Pour conclure la preuve de \ref{conj}, il reste à exhiber un sous-faisceau de $\mathcal{M}$ comme dans \ref{loctri}. C'est dû au fait général suivant:
\begin{lemme}\label{malgrange}
Soit $\mathcal{M}$ une connexion sur  $\mathds{A}^{n}_{\mathds{C}}=\Spec \mathds{C}[x_{1},\dots,x_{n}]$ méromorphe le long du diviseur $D$ donné par $x_{n}=0$. On suppose que le rang de Katz générique $\rho$ de $\mathcal{M}$ est un entier.  Alors $\mathcal{M}$ est engendré comme $\mathcal{D}_X$-module par un sous-faisceau $\mathcal{O}_X$-cohérent stable par les $x_i^{\rho+\delta_{in}}\partial_{x_i}$, $i=1,\dots , n$.
\end{lemme}
\begin{proof}
Soit $\mathcal{R}_{\tau}(\mathcal{M})$ un réseau de Malgrange \ref{malgrange2} pour une section $\tau$ de $\mathds{C}\longrightarrow \mathds{C}/\mathds{Z}$. On a $\mathcal{M}=\mathcal{O}_X[x^{-1}_n]\mathcal{R}_{\tau}(\mathcal{M})$, donc par argument de noethérianité, le $\mathcal{D}$-module engendré par $x^{-k}_n\mathcal{R}_{\tau}(\mathcal{M})$ est égal à $\mathcal{M}$ pour un entier $k_0$ assez grand. Montrons que le faisceau cohérent $x^{-k_0}_n\mathcal{R}_{\tau}(\mathcal{M})$ convient. \\ \indent 
Si $i<n$ et $m\in \mathcal{R}_{\tau}(\mathcal{M})$, on a
$$
x_i^{\rho}\partial_{x_i}(x^{-k_0}_n m)=x^{-k_0}_n(x_i^{\rho}\partial_{x_i}m)
$$
et
$$
x_n^{\rho+1}\partial_{x_n}(x^{-k_0}_n m)=-k_0x_n^{\rho}(x_n^{-k_0}m)+x_n^{-k_0}(x_n^{\rho+1}\partial_{x_n}m)
$$
donc il suffit de montrer que $\mathcal{R}_{\tau}(\mathcal{M})$ est stable par $x_n^{\rho+1}\partial_{x_n}$ et les $x_i^{\rho}\partial_{x_i} , i<n$.  \\ \indent 
D'après \eqref{Rtauinter}, il suffit de vérifier que $\mathcal{R}_{\tau}(\mathcal{M}^{\an})$ est stable par $x_n^{\rho+1}\partial_{x_n}$ et les $x_i^{\rho}\partial_{x_i} , i<n$. Du fait de 
$$
\mathcal{R}_{\tau}(\mathcal{M}^{\an}):=\mathcal{M}^{\an}\cap \mathcal{R}_{\tau}(\hat{\mathcal{M}^{\an}})$$
il suffit montrer que  $\mathcal{R}_{\tau}(\hat{\mathcal{M}^{\an}})$ est stable par $x_n^{\rho+1}\partial_{x_n}$ et les $x_i^{\rho}\partial_{x_i} , i<n$.\\ \indent
Soit $U$ un ouvert donné par le théorème \ref{malgrange1}  appliqué à $\hat{\mathcal{M}^{\an}}$. Par construction, une section de $\mathcal{M}^{\an}$ est une section de $\mathcal{R}_{\tau}(\hat{\mathcal{M}^{\an}})$ dès que sa restriction à $U$ l'est. Donc pour prouver la stabilité de $\mathcal{R}_{\tau}(\hat{\mathcal{M}^{\an}})$ par $x_n^{\rho+1}\partial_{x_n}$ et les $x_i^{\rho}\partial_{x_i} , i<n$, il suffit de se placer au-dessus de $U$. \\ \indent
Localement sur $U$, on a 
$$
f_p:(x_1,\dots , x_{n-1},t)\longrightarrow (x_1,\dots , x_{n-1},t^p)
$$
tel que
\begin{equation}\label{deLT}
f_p^{+}\hat{\mathcal{M}^{\an}}\simeq \displaystyle{\bigoplus_{i=1}^{n}} \mathcal{E}^{\phi_i}\otimes \mathcal{R}_{\phi_i}
\end{equation}
où $\mathcal{E}^{\phi_i}$ désigne $(\hat{\mathcal{O}_{X}}, d+d\phi_i)$ avec $\phi_i\in \mathcal{O}_{D,O}[t^{-1}]$ et $\mathcal{R}_{\phi_i}$ une connexion méromorphe à singularité régulière le long de $D$. Par définition
$$
\mathcal{R}_{\tau}(\hat{\mathcal{M}^{\an}})=\hat{\mathcal{M}^{\an}}\cap \mathcal{R}_{p,\tau}(\hat{\mathcal{M}^{\an}})
$$
où l'intersection doit se comprendre dans $f^{+}_p\hat{\mathcal{M}^{\an}}$, et où $\mathcal{R}_{p,\tau}(\hat{\mathcal{M}^{\an}})$ est un réseau de $f^{+}_p\hat{\mathcal{M}^{\an}}$ induisant la décomposition \eqref{deLT}. \\ \indent
Puisque $\rho$ est la plus grande pente générique de  $\mathcal{M}$, l'ordre du pôle d'une fonction $\phi_i$ intervenant dans \eqref{deLT} est $\leq p\rho$. Par construction, $\mathcal{R}_{p,\tau}(\hat{\mathcal{M}^{\an}})$ est stable par $t^{p\rho}\partial_{x_i}, i=0,\dots ,n-1$ et $t^{p\rho+1}\partial_{t}$.\\ \indent
Pour $m\in \mathcal{R}_{\tau}(\hat{\mathcal{M}^{\an}})$, on a donc
$$
x_n^{\rho}\partial_{x_i}m=t^{p\rho}\partial_{x_i}m \in \mathcal{R}_{p,\tau}(\hat{\mathcal{M}^{\an}})
$$
et du fait de $\partial_{t}=pt^{p-1}\partial_{x_n}$ 
$$
x_n^{\rho+1}\partial_{x_n}m=t^{p\rho+1}\partial_{t}m/p  \in \mathcal{R}_{p,\tau}(\hat{\mathcal{M}^{\an}})
$$
\end{proof}
\section{Rappels sur les réseaux de Malgrange}
La référence pour cette section est \cite{Mal96}. Soit $X$ une variété analytique complexe lisse $X$, $Z$  une hypersurface de $X$ et $\mathcal{M}$ une connexion méromorphe formelle à pôles  le long de $Z$, à savoir un $\hat{\mathcal{O}_{X}}(\ast Z)$-module\footnote{$\hat{\mathcal{O}_{X}}$ désigne la formalisation de $\mathcal{O}_{X}$ le long de l'idéal défini par $Z$.} localement libre de rang fini muni d'une connexion méromorphe à pôles le long de $Z$. \\ \indent
Pour $\phi \in \hat{\mathcal{O}_{X}}$, on note  $\mathcal{E}^{\phi}$ la connexion méromorphe formelle $(\hat{\mathcal{O}_{X}}, d+d\phi)$. 
\begin{definition}\label{decompadmiss}
On dit que $\mathcal{M}$ admet une \textit{décomposition admissible} en $x$ point de lissité de $Z$, si au voisinage de $x$, la connexion $\mathcal{M}$ se décompose en une somme directe finie de connexions méromorphes formelles du type $\mathcal{E}^{\phi}\otimes \mathcal{R}_{\phi}$ avec $\phi \in \hat{\mathcal{O}_{X}}(\ast Z)$ et $\mathcal{R}_{\phi}$ à singularité régulière le long de $Z$.
\end{definition}
Dans une décomposition admissible, on peut remplacer  un facteur $\phi$ par $\phi+f$ avec $f \in \hat{\mathcal{O}_{X}}$. Quitte à regrouper certains termes, on peut donc choisir les  $\mathcal{E}^{\phi_i}\otimes \mathcal{R}_{\phi_i}$, $i=1,\dots ,n$ tels que $\phi_i-\phi_j$ soit non nulle et admette un pôle le long de $Z$ pour $i\neq j$. Alors la décomposition \ref{decompadmiss} est unique. C'est celle qui sera considérée dans toute la suite.
Le premier pas vers la construction des réseaux de Malgrange est le 
\begin{theorem}\label{malgrange1}
Il existe un ouvert $U$ du lieu de lissité de $Z$ avec $S:=Z\setminus U$ fermé\footnote{Ceci signifie que tout point de $S$ est inclus localement dans un fermé analytique de $Z$ de codimension au moins $1$.} de codimension $1$ de $Z$ et tel que pour tout $x\in U$, on dispose d'un système de coordonnées locales $(x_1,\dots ,x_n)$ dans lequel $Z$ est défini par $x_n=0$ et d'un entier $p$ tel que si $f_p$ désigne l'application $(x_1,\dots ,x_{n-1},t)\longrightarrow (x_1,\dots ,x_{n-1},t^p)$, alors la connexion $f_p^{+}\mathcal{M}$ admet une décomposition admissible au sens de \ref{decompadmiss}.
\end{theorem}
Dans la suite, on garde les notations de \ref{malgrange1}. Soit $\tau$ une section de $\mathds{C}\longrightarrow \mathds{C}/\mathds{Z}$. Pour un point $x\in U$,
soit
$$
f_p^{+}\mathcal{M}\simeq \displaystyle{\bigoplus_{i=1}^{n}} \mathcal{E}^{\phi_i}\otimes \mathcal{R}_{\phi_i}
$$
la décomposition admissible de $\mathcal{M}$ au-dessus d'un voisinage de $x$ inclus dans $U$. Notons $\mathcal{R}_{p,\tau}$ le réseau de $f_p^{+}\mathcal{M}$ somme directe des réseaux de Deligne  \cite{Del} associés à $\tau$ des connexions régulières $\mathcal{R}_{\phi_i}$. Posons $\mathcal{R}_{\tau}:=\mathcal{R}_{p,\tau}\cap  \mathcal{M}$. On vérifie que $\mathcal{R}_{\tau}$ ne dépend ni de $p$ ni du choix de coordonnées $(x_1,\dots,x_n)$, et on obtient ainsi un réseau bien défini de $\mathcal{M}$ au-dessus de $U$. \\ \indent
Suivant Malgrange\footnote{Voir ce qui suit \cite[3.3.1]{Mal96}.}, notons  $\mathcal{R}_{\tau}(\mathcal{M})$ le sous-faisceau en $\hat{\mathcal{O}_{X}}$-module de $\mathcal{M}$ des sections de  $\mathcal{M}$ dont la restriction en dehors de $S$ défini une section de $\mathcal{R}_{\tau}$. On a alors \cite[3.3.1]{Mal96}
\begin{theorem}\label{malgrange2}
Le faisceau $\mathcal{R}_{\tau}(\mathcal{M})$ est cohérent, et on a $\mathcal{M}=\hat{\mathcal{O}_{X}}(\ast Z)\mathcal{R}_{\tau}(\mathcal{M})$.
\end{theorem}
On peut transporter le travail de Malgrange dans le contexte algébrique de la façon suivante: soit $X$ une variété algébrique lisse, et soit $\mathcal{M}$ une connexion méromorphe sur $X$ à pôles le long d'une hypersurface $Z$ de $X$. Soit $j:X\hookrightarrow \overline{X}$ une compactification lisse de $X$ tel que $D=\overline{X}\setminus X$ est un diviseur à croisements normaux, et soit $\tau$ une section de $\mathds{C}\longrightarrow \mathds{C}/\mathds{Z}$. Notons $\overline{Z}$ l'adhérence de $Z$ dans $\overline{X}$.\\ \indent
Alors, $(j_{\ast}\mathcal{M})^{\an}$ est une connexion analytique sur $\overline{X}$ méromorphe le long de $D\cup \overline{Z}$. Posons $$\hat{(j_{\ast}\mathcal{M})^{\an}}:=\hat{\mathcal{O}_{\overline{X}}}(\ast (D\cup \overline{Z}))\otimes (j_{\ast}\mathcal{M})^{\an}$$
D'après \cite[1.2]{Mal96}, le sous-faisceau de $(j_{\ast}\mathcal{M})^{\an}$ défini par
$$
\mathcal{R}_{\tau}((j_{\ast}\mathcal{M})^{\an}):=(j_{\ast}\mathcal{M})^{\an}\cap \mathcal{R}_{\tau}(\hat{(j_{\ast}\mathcal{M})^{\an}})$$ est cohérent et vérifie 
$$
\hat{\mathcal{O}_{\overline{X}}}\otimes \mathcal{R}_{\tau}((j_{\ast}\mathcal{M})^{\an})=\mathcal{R}_{\tau}(\hat{(j_{\ast}\mathcal{M})^{\an}})
$$
et
$$(j_{\ast}\mathcal{M})^{\an}=\mathcal{O}_{\overline{X}}(\ast (D\cup \overline{Z}))\mathcal{R}_{\tau}((j_{\ast}\mathcal{M})^{\an})$$
Par  \cite{GAGA}, le sous-faisceau $\mathcal{R}_{\tau}((j_{\ast}\mathcal{M})^{\an})$ de $(j_{\ast}\mathcal{M})^{\an}$ est l'analytifié d'un sous-faisceau cohérent de $j_{\ast}\mathcal{M}$. On en déduit par restriction à $X$ un sous-faisceau cohérent de $\mathcal{M}$ noté $\mathcal{R}_{\tau}(\mathcal{M})$ et dont $\mathcal{R}_{\tau}(\mathcal{M}^{\an})$ est l'analytifié. Il est indépendant du choix de la compactification  $\overline{X}$. Pour le voir il suffit d'observer la relation
\begin{equation}     \label{Rtauinter}
\mathcal{R}_{\tau}(\mathcal{M})=\mathcal{M}\cap \mathcal{R}_{\tau}(\mathcal{M}^{\an})
\end{equation} 
qui découle de la fidèle platitude de $(\mathcal{O}_{X,x}, \mathcal{O}_{X^{\an},x})$ pour $x\in Z$ combinée au
\begin{lemme}\label{ptilemmealglcom}
Soit $A\longrightarrow B$ un morphisme fidèlement plat,  $M$ un $A$-module, et $N$ un sous-module de $M$. Alors on a $N=M\cap (B\otimes_A N)$, où l'intersection a lieu dans $B\otimes_A M$.
\end{lemme}
\noindent
Par \eqref{Rtauinter}, on a  $\mathcal{M}=\mathcal{O}_{X}(\ast Z)\mathcal{R}_{\tau}(\mathcal{M})$.
Le faisceau $\mathcal{R}_{\tau}(\mathcal{M})$ est appelé le \textit{réseau\footnote{La terminologie de réseau est trompeuse car on ne sait pas a priori si $\mathcal{R}_{\tau}(\mathcal{M})$ est localement libre en dimension $>2$. Voir \cite[3.3.2]{Mal96}.} de Malgrange de $\mathcal{M}$ associé à $\tau$}.

\bibliographystyle{amsalpha}
\bibliography{AScaslisse}

\end{document}